\documentclass[12pt]{amsart}
\usepackage{amssymb,amsmath,amsthm,mathrsfs,multirow,xcolor,framed,url}
\usepackage[pdftex,
         pdfauthor={Dandan Chen, Rong Chen and Siyu Yin},
         pdftitle={On the total number of ones associated with cranks of partitions modulo 11},
         pdfsubject={MATHEMATICS},
         pdfkeywords={Total number of ones, Crank for partitions, Theta-functions.},
         pdfproducer={Latex with hyperref},
         pdfcreator={pdflatex}]{hyperref}
\usepackage [latin1]{inputenc}
\newtheorem{theorem}{Theorem}[section]
\newtheorem{lemma}[theorem]{Lemma}

\theoremstyle{definition}
\newtheorem{definition}[theorem]{Definition}

\theoremstyle{remark}

\numberwithin{equation}{section}

\newcommand\nutwid{\overset {\text{\lower 3pt\hbox{$\sim$}}}\nu}

\allowdisplaybreaks

\newcommand\omycite[1]{}

\newcommand{\beqs}{\begin{equation*}}
\newcommand{\eeqs}{\end{equation*}}
\newcommand{\beq}{\begin{equation}}
\newcommand{\eeq}{\end{equation}}
\renewcommand{\MR}[1]{\href{http://www.ams.org/mathscinet-getitem?mr={#1}}{MR{#1}}}

\begin{document}
\title[Congruences modulo powers of $3$ for $6$-colored generalized Frobenius partitions ]{Congruences modulo powers of $3$ for $6$-colored generalized Frobenius partitions }


\author{Dandan Chen}
\address{Department of Mathematics, Shanghai University, People's Republic of China}
\address{Newtouch Center for Mathematics of Shanghai University, Shanghai, People's Republic of China}
\email{mathcdd@shu.edu.cn}
\author{Siyu Yin*}
\address{Department of Mathematics, East China Normal University, People's Republic of China}
\email{siyuyin@shu.edu.cn, siyuyin0113@126.com}



\date{}



\begin{abstract}
In $1984$, Andrews introduced the family of partition functions $c\phi_k(n)$, which enumerate generalized Frobenius partitions of $n$ with $k$ colors. In $2016$, Gu, Wang, and Xia established several congruences for $c\phi_6(n)$ and proposed a conjecture concerning congruences modulo powers of $3$ for this function. In this paper, we resolve a revised version of their conjecture by employing an approach analogous to that developed by Banerjee and Smoot.
\end{abstract}

\maketitle


\section{Introduction}
A partition of a positive integer $n$ is a non-increasing sequence of positive integers that sum to $n$. The number of such partitions is denoted by $p(n)$. Among the most celebrated results in partition theory are the following congruences, discovered and proved by Ramanujan \cite{Ramanujan-1919}:
\begin{align*}
p(5n+4)\equiv 0 &\pmod 5,\\
p(7n+5)\equiv 0 &\pmod 7,\\
p(11n+6)\equiv 0 &\pmod {11},
\end{align*}
for all non-negative integers $n$. In fact, for any integer $\alpha\geq1$,
\begin{align}
p(5^{\alpha}n+\delta_{5,\alpha})&\equiv 0\pmod {5^{\alpha}},\label{5-congruence}\\
p(7^{\alpha}n+\delta_{7,\alpha})&\equiv 0\pmod {7^{[\frac{\alpha+2}{2}]}},\label{7-congruence}\\
p(11^{\alpha}n+\delta_{11,\alpha})&\equiv 0\pmod {11^{\alpha}}\label{11-congruence},
\end{align}
where $\delta_{t,\alpha}$ is defined by the congruence $24\delta_{t,\alpha}\equiv1\pmod {t^{\alpha}}$. \eqref{5-congruence} and \eqref{7-congruence} were first proved by Watson \cite{Watson-1938} in 1938. Later,  elementary proofs of \eqref{5-congruence} and \eqref{7-congruence} were given by  Hirschhorn and Hunt \cite{Hirschhorn-1981} and  Garvan \cite{Garvan-1984}, respectively. The case for $11$, \eqref{11-congruence} was established by Atkin \cite{Atkin-1967} in 1967.

In his 1984 AMS Memoir \cite{Andrews-1984}, Andrews introduced the concept of generalized Frobenius partitions of $n$. It is defined as a two-rowed array of the form
\begin{align*}
\left(\begin{array}{cccc} a_1 & a_2 & \cdots & a_r \\
b_1 & b_2 & \cdots &b_r \end{array}\right)
\end{align*}
where  the entries $a_i$, $b_i$  are non-negative integers satisfying the condition
\begin{align*}
n=r+\sum_{i=1}^ra_i+\sum_{i=1}^rb_i.
\end{align*}
Furthermore, both sequences $(a_1,a_2,\cdots,a_r)$ and $(b_1,b_2,\cdots,b_r)$ must be non-increasing.

Andrews \cite{Andrews-1984} also introduced a variant known as the $k$-colored generalized Frobenius partition. It is an array of the form described above, but whose entries are drawn from $k$ distinct copies of the non-negative integers, each copy distinguished by a color. The entries in each row are ordered first by their numerical value and subsequently by their color, with the additional constraint that no two consecutive entries in any row are identical in both value and color.

 For any positive integer $k$, let $c\phi_k(n)$
denote the number of such $k$-colored generalized Frobenius partitions of $n$. Andrews \cite{Andrews-1984} established the following  generating function:
\begin{align*}
\sum_{n=0}^{\infty}c\phi_k(n)q^n=\prod_{n=1}^{\infty}\frac{1}{(1-q^n)^k}\sum_{m_1,m_2,\cdots,m_{k-1}=-\infty}^{\infty}q^{Q(m_1,m_2,\cdots,m_{k-1})},
\end{align*}
where
\begin{align*}
Q(m_1,m_2,\cdots,m_{k-1})=\sum_{j=1}^{k-1}m_j^2+\sum_{1\leq i\textless j\leq k-1}m_im_j.
\end{align*}

This paper focuses on generalized $6$-colored Frobenius partitions. Baruah and Sarmash \cite[Theorem 2.1]{Baruah-Sarmash-2015} expressed the generating function for $c\phi_6(n)$ in terms of Ramanujan's theta functions as follows:
\begin{align}
\sum_{n=0}^{\infty}c\phi_6(n)q^n=&\frac{1}{(q;q)_{\infty}^6}\{\varphi^3(q)\varphi(q^2)\varphi(q^6)+24q\psi^3(q)\psi(q^2)\psi(q^3)\\
\nonumber&+4q^2\varphi^3(q)\psi(q^4)\psi(q^{12})\},
\end{align}
where
\begin{align*}
\varphi(q)=(-q;q^2)_{\infty}^2(q^2;q^2)_{\infty},\\
\psi(q)=(-q;q^2)_{\infty}(q^4;q^4)_{\infty}.
\end{align*}
They \cite{Baruah-Sarmash-2015} also conjectured that for all $n\geq0$,
\begin{align}\label{B-S-con}
c\phi_6(3n+2)\equiv0\pmod{27}.
\end{align}
This conjecture was proved by Xia \cite{Xia-2015} using the generating function for $c\phi_6(3n+2)$ given by Baruah and Sarmah. Xia further conjectured that
\begin{align*}
c\phi_6(9n+7)\equiv0\pmod{27},
\end{align*}
which was subsequently established by Hirschhorn  \cite{Hirschhorn-2015}.

 In 2016, Gu, Wang, and Xia \cite{Gu-Wang-Xia-2016} proved several congruences modulo powers of 3 for $c\phi_6(n)$ by employing the generating function for $c\phi_6(3n+1)$ provided by Hirschhorn \cite{Hirschhorn-2015}. For example, they showed that for all $n\geq0$,
\begin{align*}
c\phi_6(19683n+11482)&\equiv0\pmod{3^7},\\
c\phi_6(59049n+44287)&\equiv0\pmod{3^7}.
\end{align*}
Additonally, they proposed the following conjecture \cite[Conjecture 1.2]{Gu-Wang-Xia-2016}: for any integer $j\geq8$, there exist positive integers $a$ and $k$ such that for all $n\geq0$,
\begin{align}
c\phi_6(3^kn+a)\equiv0\pmod{3^j}.
\end{align}
Tang \cite[Conjecture 1.4]{Tang} confirmed the conjecture in the following explicit form, though without providing a proof:
\begin{theorem}\label{main-result}
For integers $n\geq0$, and $\alpha\geq1$,
\begin{align}
c\phi_6(3^{\alpha}n+\lambda_{\alpha})\equiv0\pmod{3^{\lfloor\frac{\alpha}{2}\rfloor+2}},
\end{align}
where
\begin{align*}
\lambda_{2\alpha-1}=\frac{3^{2\alpha-1}+1}{4}      ~\text{and}~
\lambda_{2\alpha}=\frac{3^{2\alpha+1}+1}{4}.
\end{align*}
\end{theorem}

Our proof follows an approach similar to Banerjee and Smoot \cite{Banerjee-Smoot-2025}, utilizing the recently developed localization method and internal algebraic structure. The paper is organized as follows. In Section \ref{sec-pre} we introduce the necessary notations and definitions. In Section \ref{sec-lemma} we derive modular equations on $\Gamma_0(12)$. Section \ref{sec-main} presents the proof of Theorem \ref{main-result}.

\section{Preliminaries}\label{sec-pre}

In this paper, we adhere to the following conventions. The sets of integers, nonegative integers, and positive integers are denoted by $\mathbb{Z}$, $\mathbb{N}=\{0,1,2,\cdots\}$, and $\mathbb{N^\ast}=\{1,2,\cdots\}$, respectively. Let $\mathbb{H}:=\{\tau\in\mathbb{C}:\Im(\tau)\textgreater 0\}$ represent the complex upper half plane. We use the notation $\eta_n(\tau):=\eta(n\tau),~~n\in\mathbb{Z},~\tau\in\mathbb{H}$. For $x\in\mathbb{R}$, $\lfloor x\rfloor$ denotes the largest integer less than or equal to $x$, and $\lceil x\rceil$ denotes the smallest integer greater than or equal to $x$. For $n\in\mathbb{Z}$, $\pi(n)$ denotes  the 3-adic order of n (i.e. the highest power of 3 that divides $n$).

 Let $f=\sum_{n\in\mathbb{Z}}a_nq^n$ (with $f\neq0$) be a series such that  $a_n=0$ for all but finitely many $n\textless 0$. We define the order of $f$ (with respect to $q$) as the smallest integer $N$ for which $a_N\neq0$, and write $ord_q(f)=N$. Furthermore, if $F=f\circ t=\sum_{n\in\mathbb{Z}}a_nt^n$, where $t=\sum_{n\geq 1}b_nq^n$, then the $t$-order of $F$ is defined as the smallest integer $N$ such that $a_N\neq 0$, denoted  $N=ord_t(F)$.


\begin{definition}
For $f:\mathbb{H}\rightarrow\mathbb{C}$ and $m\in \mathbb{N}^{\ast}$ we define $U_m(f):\mathbb{H} \rightarrow \mathbb{C}$ by
\begin{align*}
U_m(f)(\tau):=\frac{1}{m}\sum_{\lambda=0}^{m-1}f(\frac{\tau+\lambda}{m}),~\tau\in\mathbb{H}.
\end{align*}
\end{definition}
$U_m$ is linear (over $\mathbb{C}$); in addition, it is easy to verify that
\begin{align*}
U_{mn}=U_m\circ U_n=U_n\circ U_m,~m,n\in\mathbb{N}^{\ast}.
\end{align*}
The operators $U_m$, introduced by Atkin and Lehner \cite{Atkin-Lehner-1970}, are closely related to Hecke operators. They typically arise in the context of partition congruences \cite{Andrews-1976} mostly because of the property: if
\begin{align*}
f(\tau)=\sum_{n=-\infty}^{\infty}f_nq^n~~(q=e^{2\pi i\tau}),
\end{align*}
then
\begin{align*}
U_m(f)(\tau)=\sum_{n=-\infty}^{\infty}f_{mn}q^n.
\end{align*}

\begin{definition}\label{def-A-B}
For $f:\mathbb{H}\rightarrow\mathbb{C}$ we define $U_A(f), U_B(f):\mathbb{H} \rightarrow \mathbb{C}$ by $U_A(f):=U_3(Af)$ and $U_B(f):=U_3(Bf)$, where
\begin{align*}
A:=\frac{\eta_9^9\eta_4^2\eta_2^5}{\eta_{36}^2\eta_{18}^5\eta_1^9}~~\text{and}~~ B:=\frac{\eta_9\eta_2^2}{\eta_{18}^2\eta_1}.
\end{align*}
\end{definition}

\begin{definition}
\label{r:defL}
We define the $U$-sequence $(L_{\alpha})_{\alpha\geq0}$ by
\begin{align*}
L_0:=\frac{\eta_{12}^5\eta_3\eta_2^8}{\eta_{24}^2\eta_8^2\eta_6^4\eta_4^3\eta_1^3}+24+4\frac{\eta_{24}^2\eta_8^2\eta_3\eta_2^{10}}{\eta_{12}\eta_6^2\eta_4^9\eta_1^3} \end{align*}
and for $\alpha\geq 1$:
\begin{align*}
L_{2\alpha-1}:=U_A(L_{2\alpha-2})~~~~~\text{and}~~~L_{2\alpha}:=U_B(L_{2\alpha-1}).
\end{align*}
\end{definition}

The proof of the following lemma is completely analogous to \cite[p. 23]{Atkin-1967} and we omit it.
\begin{lemma}\label{lem-L-alpha}
For $\alpha \in \mathbb{N}^{\ast}$, we have
\begin{align*}
&L_{2\alpha-1}=q^{-1}\prod_{n=1}^{\infty}\frac{(1-q^n)(1-q^{3n})^9}{(1-q^{2n})^2(1-q^{6n})^5(1-q^{12n})^2}\sum_{n=0}^{\infty}c\phi_6(3^{2\alpha-1}n+\lambda_{2\alpha-1})q^n,\\
&L_{2\alpha}=\prod_{n=1}^{\infty}\frac{(1-q^{n})^9(1-q^{3n})}{(1-q^{2n})^5(1-q^{4n})^2(1-q^{6n})^2}\sum_{n=0}^{\infty}c\phi_6(3^{2\alpha}n+\lambda_{2\alpha})q^n.
\end{align*}
\end{lemma}

\begin{definition}\label{t-p-y}
Let $t, y, p_0, p_1$ be functions defined on $\mathbb{H}$ as follows:
\begin{align*}
t:=\frac{\eta_{12}^4\eta_2^2}{\eta_6^2\eta_4^4},~~~y:=\frac{\eta_4^3\eta_3}{\eta_{12}\eta_1^3},~~~p_0:=\frac{\eta_{12}^4\eta_3^{12}\eta_2^8}{\eta_6^8\eta_4^{12}\eta_1^4},~~~p_1:=\frac{\eta_{12}^2\eta_3^6\eta_2^4}{\eta_6^4\eta_4^6\eta_1^2},
\end{align*}
 which have Laurent series expansions in powers of $q$ with coefficients in $\mathbb{Z}$.
\end{definition}

\section{Modular equation}
\label{sec-lemma}

Our proof of the Gu-Wang-Xia conjecture relies on the modular identities in the Appendix.  Many of these, specifically those in Groups $I$-$III$, can be automatically verified using Garvan's MAPLE package ETA (see \eqref{r:eta} and \cite{gtutorial})
\begin{align}
\label{r:eta}
https://qseries.org/fgarvan/qmaple/ETA/
\end{align}
For example,  the expression $L_0=t^{-1}+27+3t+9t^2$, where $L_0$ is defined  in  Definition \ref{r:defL}, can be directly verified using this package.
The package also yields the identity for $L_1$:
\begin{align}\label{L-1}
L_1=&y^8p_0(4\cdot3^2t^{-1}+71\cdot3^4+2351\cdot3^3t+89\cdot3^5t^2-1975\cdot3^4t^3+407\cdot3^5t^4\\
\nonumber&-19\cdot3^7t^5+11\cdot3^7t^6-3^8t^7).
\end{align}

\begin{definition}\label{def-aj-t}
Let $t=t(\tau)$ be as in Definition \ref{t-p-y}. We define:
\begin{align*}
a_0(t)=-t,~~~a_1(t)=3t^2,~~~a_2(t)=-9t^3+6t.
\end{align*}
Define $s:\{0,1,2\}\times\{1,2,3\}\rightarrow\mathbb{Z}$ to be the unique function satisfying
\begin{align}\label{def-aj}
a_j(t)=\sum_{l=1}^3 s(j,l)3^{\lfloor\frac{2l+j}{3}\rfloor}t^l.
\end{align}
\end{definition}

\begin{lemma}\label{t-aj}
For $\lambda\in\{0,1,2\}$, let
\begin{align*}
t_{\lambda}(\tau):=t(\frac{\tau+\lambda}{3}),~~\tau\in\mathbb{H}.
\end{align*}
Then in the polynomial ring $\mathbb{C}(t)[X]$:
\begin{align*}
X^3+a_2(t)X^2+a_1(t)X+a_0(t)=(X-t_0)(X-t_1)(X-t_2).
\end{align*}
\end{lemma}
\begin{proof}
We first  prove that
\begin{align*}
\prod_{\lambda=0}^2t_{\lambda}(\tau)=-a_0(t)=t.
\end{align*}
Let $\omega:=e^{\frac{2\pi i}{3}}$. Then
\begin{align*}
\prod_{\lambda=0}^2t_{\lambda}(\tau)&=\prod_{\lambda=0}^2q^{\frac{1}{3}}\omega^{\lambda}\prod_{n=1}^{\infty}\frac{(1-q^{4n})^4(1-\omega^{2\lambda n}q^{\frac{2n}{3}})^2}{(1-q^{2n})^2(1-\omega^{4\lambda n}q^{\frac{4n}{3}})^4}\\
&=q\prod_{n=1}^{\infty}\prod_{\lambda=0}^2\frac{(1-q^{4n})^4(1-\omega^{2\lambda n}q^{\frac{2n}{3}})^2}{(1-q^{2n})^2(1-\omega^{4\lambda n}q^{\frac{4n}{3}})^4}\\
&=q\prod_{n=1}^{\infty}\frac{(1-q^{4n})^{12}}{(1-q^{2n})^6}\prod_{n=1}^{\infty}\frac{(1-q^{2n})^2(1-q^{12n})^4}{(1-q^{6n})^2(1-q^{4n})^4}\prod_{n=1}^{\infty}\frac{(1-q^{2n})^{6}}{(1-q^{4n})^{12}}\\       &=t,
\end{align*}
where we  use the fact that $\prod_{\lambda=0}^2(1-\omega^{\lambda n}z)$ equals $(1-z)^3$ if $3\mid n$, and $1-z^3$ otherwise.
For $a_1(t)$, we have
\begin{align*}
a_1(t)=&t_0(\tau)t_1(\tau)+t_0(\tau)t_2(\tau)+t_1(\tau)t_2(\tau)
=3tU_3(t^{-1})
=3t^2,
\end{align*}
where the last equation was verified using the ETA MAPLE package.
For $a_2(t)$, we have
\begin{align*}
a_2(t)=&-t_0(\tau)-t_1(\tau)-t_2(\tau)
=-3U_3(t)
=-9t^3+6t,
\end{align*}
where the last equation was again verified using the ETA MAPLE package.
\end{proof}

Similarly, we can state the following definition and lemma without detailed proofs.
\begin{definition}\label{def-bj-y}
With $y=y(\tau)$ as in Definition \ref{t-p-y} we define:
\begin{align*}
b_0(y)=-y^3,~~~b_1(y)=8y^3-3y^2-3y+1,~~~b_2(y)=-16y^3+12y^2+3y-2.
\end{align*}
\end{definition}

\begin{lemma}\label{y-bj}
For $\lambda\in\{0,1,2\}$, let
\begin{align*}
y_{\lambda}(\tau):=y(\frac{\tau+\lambda}{3}),~~\tau\in\mathbb{H}.
\end{align*}
Then in the polynomial ring $\mathbb{C}(y)[X]$:
\begin{align*}
X^3+b_2(y)X^2+b_1(y)X+b_0(y)=(X-y_0)(X-y_1)(X-y_2).
\end{align*}
\end{lemma}

\begin{lemma}\label{U2-j-jl2}
For $u:\mathbb{H}\rightarrow\mathbb{C}$, $j\in\mathbb{Z}$, and $l\in\mathbb{N}$ :
\begin{align}
\label{re-t}&U_3(ut^{j+3})=-a_0(t)U_3(ut^{j})-a_1(t)U_3(ut^{j+1})-a_2(t)U_3(ut^{j+2});\\
\label{re-y}&U_3(uy^{l+3})=-b_0(y)U_3(uy^l)-b_1(y)U_3(uy^{l+1})-b_2(y)U_3(uy^{l+2}).
\end{align}
\end{lemma}
\begin{proof}
For $\lambda\in\{0,1,2\}$, Lemma \ref{t-aj} implies
\begin{align*}
t_{\lambda}^3+a_2(t)t_{\lambda}^2+a_1(t)t_{\lambda}+a_0(t)=0.
\end{align*}
Multiplying both sides by $u_{\lambda}t_{\lambda}^{j}$, where $u_{\lambda}(\tau):=u((\tau+\lambda)/3)$, yields
\begin{align*}
u_{\lambda}t_{\lambda}^{j+3}+a_2(t)u_{\lambda}t_{\lambda}^{j+2}+a_1(t)u_{\lambda}t_{\lambda}^{j+1}+a_0(t)u_{\lambda}t_{\lambda}^{j}=0.
\end{align*}
Summing both sides over all $\lambda\in\{0,1,2\}$ completes the proof of \eqref{re-t}. The identity \eqref{re-y} can be proved similarly.
\end{proof}

In the following proofs, we need to use both \eqref{re-t} and \eqref{re-y}. However, the coefficients of $y^mt^k$ are not easily described, where $m\in\mathbb{N}$ and $k\in\mathbb{Z}$. According to Group II in the Appendix,  by Lemma \ref{order-2}, it is possible to assume that $U_A(p_1y^mt^k)$ can be written in the form $y^{3m+8}p_0\sum_na_nt^n$ with $a_n\in\mathbb{Z}$. Based on the identity $y=\frac{1}{1-3t}$, we can rewrite $b_0(y)$, $b_1(y)$, and $b_2(y)$ as:
\begin{align}
\label{b-y-0}b_0(y)&=-y^9(1-2\cdot3^2t+5\cdot3^3t^2-20\cdot3^3t^3+5\cdot3^5t^4-2\cdot3^6t^5+3^6t^6);\\
\label{b-y-1}b_1(y)&=-y^6(-3+3^2t+3^4t^2-14\cdot3^3t^3+3^5t^4+3^6t^5-3^6t^6);\\
\label{b-y-2}b_2(y)&=-y^3(3+4\cdot3^2t+3^3t^2-2\cdot3^3t^3).
\end{align}
Similarly, by Lemma \ref{order-2}, we have $U_B(p_0y^mt^k)=y^{3m}p_1\sum_na_nt^n$. According to Lemma \ref{U2-j-jl2}, we can calculate $U_A(p_1y^mt^k)$ and $U_B(p_0y^mt^k)$ using the identities in Appendix for all  $k\geq -1$ and $m\in\mathbb{N}$ recursively. For example,
\begin{align*}
U_A(p_1t^2)=&-\sum_{j=0}^2a_j(t)U_A(p_1t^{j-1})\\
=&y^8p_0(11+209\cdot3t-22\cdot3^3t^2-106\cdot3^3t^3+259\cdot3^4t^4+259\cdot3^4t^5-11\cdot3^8t^6\\
&+68\cdot3^7t^7-238\cdot3^6t^8+7\cdot3^8t^9+29\cdot3^8t^{10}-16\cdot3^9t^{11}+11\cdot3^9t^{12}-3^{10}t^{13}),      \\
U_B(p_0y^3t^{-1})=&-\sum_{l=0}^2b_l(y)U_B(p_0y^{l}t^{-1})\\
=&y^9p_1(37\cdot3+2992\cdot3t+13628\cdot3^2t^2+3872\cdot3^4t^3-4814\cdot3^4t^4-7600\cdot3^4t^5\\
&+2564\cdot3^5t^6-3^{10}t^8).
\end{align*}

\section{Proof of Theorem \ref{main-result}}
\label{sec-main}
\begin{lemma}\label{order}
Let $v, u:\mathbb{H}\rightarrow\mathbb{C}$ and $\ell\in\mathbb{Z}$. Suppose that for $\ell\leq k\leq \ell+2$,  there exist polynomials $p_k(t)\in\mathbb{Z}[t,t^{-1}]$ such that
\begin{align}\label{ord-1}
U_3(ut^k)=vp_k(t)
\end{align}
and
\begin{align}\label{ord-2}
ord_t(p_k(t))\geq \lceil\frac{k+s}{3}\rceil,
\end{align}
for some fixed integer $s$. Then there exist families of polynomials $p_k(t)\in\mathbb{Z}[t,t^{-1}]$, defined for all integers $k\in\mathbb{Z}$, such that both \eqref{ord-1} and \eqref{ord-2} hold for all $k\in\mathbb{Z}$.
\end{lemma}

\begin{proof}
Let $N\textgreater \ell+2$ be an integer and assume by induction that there exist families of polynomials $p_k(t)$ such that \eqref{ord-1} and \eqref{ord-2} hold for all $\ell\leq k\leq N-1$. Suppose
\begin{align*}
p_k(t)=\sum_{n\geq\lceil\frac{k+s}{3}\rceil}c(k,n)t^n, ~~~\ell\leq k\leq N-1,
\end{align*}
with integers $c(k,n)$. Applying Lemma \ref{U2-j-jl2}, we obtain:
\begin{align*}
U_3(ut^N)=&-\sum_{j=0}^2a_j(t)U_3(ut^{N+j-3})\\
=&-\sum_{j=0}^2a_j(t)v\sum_{n\geq\lceil\frac{N+j-3+s}{3}\rceil}c(N+j-3,n)t^n\\
=&-v\sum_{j=0}^2a_j(t)t^{-1}\sum_{n\geq\lceil\frac{N+j+s}{3}\rceil}c_i(N+j-3,n-1)t^n.
\end{align*}
Since $a_j(t)t^{-1}$  are polynomials in $t$ for $0\leq j\leq2$, this defines a polynomial $p_N(t)$ with the desired properties. The induction argument for $N\textless \ell$ is analogous.
\end{proof}

\begin{lemma}\label{U-power}
Let $v, u:\mathbb{H}\rightarrow\mathbb{C}$ and $\ell\in\mathbb{Z}$. Suppose that for $\ell\leq k\leq \ell+2$  there exist polynomials $p_k(t)\in\mathbb{Z}[t,t^{-1}]$, such that
\begin{align}\label{power-1}
U_3(ut^k)=vp_k(t)
\end{align}
with
\begin{align}\label{power-2}
p_k(t)=\sum_nc(k,n)3^{\lfloor\frac{2n-k+\gamma}{3}\rfloor}t^n,
\end{align}
for some integers $\gamma$ and $c(k,n)$. Then there exist families of polynomials $p_k(t)\in\mathbb{Z}[t,t^{-1}]$, defined for all $k\in\mathbb{Z}$, of the form \eqref{power-2} such that property \eqref{power-1} holds for  all $k\in\mathbb{Z}$.
\end{lemma}

\begin{proof}
Suppose for an integer $N\textgreater \ell+2$ there exist families of polynomials $p_k(t)$ of the form \eqref{power-2} satisfying property \eqref{power-1} for all $\ell\leq k\leq N-1$. We proceed by mathematical induction on $N$. Applying Lemma \ref{U2-j-jl2}  and using the induction hypothesis \eqref{power-1} and \eqref{power-2}, we obtain:
\begin{align*}
U_3(ut^N)=-\sum_{j=0}^2a_j(t)v\sum_nc(N+j-3,n)3^{\lfloor{\frac{2n-(N+j-3)+\gamma}{3}}\rfloor}t^n.
\end{align*}
Utilizing \eqref{def-aj},  this becomes:
\begin{align*}
U_3(ut^N)&=-\sum_{j=0}^2\sum_{\ell=1}^3s(j,l)3^{\lfloor{\frac{2\ell+j}{3}}\rfloor}t^\ell
v\sum_nc(N+j-3,n)3^{\lfloor{\frac{2n-(N+j-3)+\gamma}{3}}\rfloor}t^n\\
&=-v\sum_{j=0}^2\sum_{l=1}^3\sum_ns(j,\ell)c(N+j-3,n-\ell)3^{\lfloor{\frac{2(n-\ell)-(N+j-3)+\gamma}{3}}\rfloor+\lfloor{\frac{2\ell+j}{3}}\rfloor}t^n.
\end{align*}
The induction step is completed by simplifying the exponent of $3$ as follows:
\begin{align*}
&\lfloor{\frac{2(n-\ell)-(N+j-3)+\gamma}{3}}\rfloor+\lfloor{\frac{2\ell+j}{3}}\rfloor\\
&\geq \lfloor{\frac{2(n-\ell)-(N+j-3)+\gamma+2\ell+j-2}{3}}\rfloor\\
&\geq\lfloor{\frac{2n-N+\gamma}{3}}\rfloor.
\end{align*}
The induction argument for $N\textless \ell$ proceeds analogously.
\end{proof}

\begin{lemma}\label{order-2}
Let $v, u:\mathbb{H}\rightarrow\mathbb{C}$, $\ell\in\mathbb{Z}$, and $r\in\mathbb{N}$. Suppose that for $\ell\leq k\leq \ell+2$ and $r\leq m \leq r+2$,  there exist polynomials $p_{m,k}(t)\in\mathbb{Z}[t,t^{-1}]$ such that
\begin{align}\label{ord-2-1}
U_3(uy^mt^k)=vy^{3m+a}p_{m,k}(t)
\end{align}
and
\begin{align}\label{ord-2-2}
ord_t(p_{m,k}(t))\geq \lceil\frac{k+s}{3}\rceil,
\end{align}
for some fixed integers $a$ and $s$. Then there exist families of polynomials $p_{m,k}(t)\in\mathbb{Z}[t,t^{-1}]$, defined for all integers  $k\in\mathbb{Z}$, such that both \eqref{ord-2-1} and \eqref{ord-2-2} hold for all $m\geq r$ and $k\in\mathbb{Z}$.
\end{lemma}
\begin{proof}
By Lemma \ref{order}, we conclude that for $r\leq m\leq r+2$, both \eqref{ord-2-1} and \eqref{ord-2-2} hold for all $k\in\mathbb{Z}$. Now fix an integer $k$, and let $N\textgreater r+2$ be an integer. Assume by induction that there exist families of polynomials $p_{m,k}(t)$ such that \eqref{ord-2-1} and \eqref{ord-2-2} hold for all $r\leq m\leq N-1$. Suppose
\begin{align*}
p_{m,k}(t)=\sum_{n\geq\lceil\frac{k+s}{3}\rceil}c(k,m,n)t^n, ~~~r\leq m\leq N-1,
\end{align*}
with integers $c(k,m,n)$. Applying Lemma \ref{U2-j-jl2}, we obtain:
\begin{align*}
U_3(uy^Nt^k)=&-\sum_{j=0}^2b_j(y)U_3(uy^{N+j-3}t^k)\\
=&-\sum_{j=0}^2b_j(y)vy^{3(N+j-3)+a}\sum_{n\geq\lceil\frac{k+s}{3}\rceil}c(k,m,n)t^n.
\end{align*}
Using equations \eqref{b-y-0}--\eqref{b-y-2}, we verify that both \eqref{ord-2-1} and \eqref{ord-2-2} hold for $m=N$. By induction, it follows that these properties hold for all $m\geq r$, which completes the proof.
\end{proof}

\begin{lemma}\label{U-power-2}
Let $v, u:\mathbb{H}\rightarrow\mathbb{C}$, $\ell\in\mathbb{Z}$, and $r\in\mathbb{N}$. Suppose that for $\ell\leq k\leq \ell+2$ and $r\leq m \leq r+2$,  there exist polynomials $p_{m,k}(t)\in\mathbb{Z}[t,t^{-1}]$ such that
\begin{align}\label{power-2-1}
U_3(uy^mt^k)=vy^{3m+a}p_{m,k}(t)
\end{align}
and
\begin{align}\label{power-2-2}
p_{m,k}(t)=\sum_nc(k,m,n)3^{\lfloor\frac{2n-k+\gamma}{3}\rfloor}t^n,
\end{align}
for some  integers $\gamma$ and $c(k,m,n)$. Then there exist families of polynomials $p_{m,k}(t)\in\mathbb{Z}[t,t^{-1}]$, defined for all $k\in\mathbb{Z}$ and  of the form \eqref{power-2-2}, such that property \eqref{power-2-1} holds for  all $m\geq r$ and $k\in\mathbb{Z}$.
\end{lemma}

\begin{proof}
By Lemma \ref{U-power}, we conclude that for $r\leq m\leq r+2$, both \eqref{power-2-1} and \eqref{power-2-2} hold for all $k\in\mathbb{Z}$. Now fix an integer $k$, and  let $N\textgreater r+2$ be an integer. Assume by induction that there exist families of polynomials $p_{m,k}(t)$ such that \eqref{power-2-1} and \eqref{power-2-2} hold for all $r\leq m\leq N-1$. We proceed by mathematical induction on N.

Applying Lemma \ref{U2-j-jl2} and using the induction hypothesis \eqref{power-2-1} and \eqref{power-2-2}, we attain:
\begin{align*}
U_3(uy^Nt^k)=-\sum_{j=0}^2b_j(y)U_3(uy^{N+j-3}t^k)=-v\sum_{j=0}^2b_j(y)y^{3(N+j-3)+a}\sum_nc(k,m,n)3^{\lfloor\frac{2n-k+\gamma}{3}\rfloor}t^n.
\end{align*}
Since each $b_j(y)$ can be expressed as $y^{9-3j}f(t)$ for $j\in\{0,1,2\}$, where $ord_t(f(t))\geq 0$, it follows that both \eqref{power-2-1} and \eqref{power-2-2} hold for  $m=N$. By in duction, this completes the proof of this lemma.
\end{proof}

\begin{definition}
A map $a:\mathbb{Z}\rightarrow\mathbb{Z}$ is called a discrete  function if it has finite support. A map $a:\mathbb{Z}\times\mathbb{N}\times\mathbb{Z}\rightarrow\mathbb{Z}$ is called a discrete array if for each $i\in\mathbb{Z}$ and $j\in\mathbb{N}$, the map $a(i,j,-):\mathbb{Z}\rightarrow\mathbb{Z},k\mapsto a(i,j,k)$ has finite support.
\end{definition}

\begin{lemma}\label{U-fundamental}
Let $U_A(f)$, $U_B(f)$ be as in Definition \ref{def-A-B}, and let $p_0, p_1$, and $y$ be as in Definition \ref{t-p-y}. Then there exist discrete arrays  $a$ and $b$ such that the following relations hold for all $m\in\mathbb{N}$ and $k\in\mathbb{Z}$:
\begin{align}
&U_A(p_1y^mt^k)=y^{3m+8}p_0\sum_{n\geq \lceil\frac{k-3}{3}\rceil}a(k,m,n)t^n,\text{where}~ \pi(a(k,m,n))\geq \small{\lfloor\frac{2n-k+3}{3}\rfloor};   \label{U-0-t}\\
&U_B(p_0y^mt^k)=y^{3m}p_1\sum_{n\geq \lceil\frac{k-1}{3}\rceil}b(k,m,n)t^n,\text{where}~\pi(b(k,m,n))\geq {\lfloor\frac{2n-k+2}{3}\rfloor} .\label{U-0-p1t}
\end{align}
\end{lemma}
\begin{proof}
The Appendix lists eighteen fundamental relations. The nine fundamental relations in Group II match the pattern of the relation \eqref{U-0-t} for three values of $m$ and three values of $k$. Similarly, the relations in Groups III correspond to the pattern of \eqref{U-0-p1t}. Applying Lemmas \ref{order}--\ref{U-power-2}, we immediately obtin the statement for all $m\in \mathbb{N}$ and $k\in\mathbb{Z}$.
\end{proof}

In the following proof, we use the next lemma to describe the powers of $y$ in $L_{\alpha}$.
\begin{lemma}\label{y-count}
Define $f(1)=8$, and  for $s\geq 2$, dedine $f(s)$ recursively by:
\begin{align*}
f(s)=\left\{\begin{array}{ll}
			3f(s-1), ~&\text{if}~s~\text{is even,}\\
			3f(s-1)+8, ~&\text{if}~s~\text{is odd.}\end{array}\right.
\end{align*}
Then for all $s\geq1$,
\begin{align*}
f(s)=\left\{\begin{array}{ll}
			3^{s+1}-3, ~&\text{if}~s~\text{is even,}\\
			3^{s+1}-1, ~&\text{if}~s~\text{is odd.}\end{array}\right.
\end{align*}
\end{lemma}

In the  Appendix, we mention the identity  $y=\frac{1}{1-3t}$. In the following proof, we will express each $L_{\alpha}$ as a rational polynomial in which the denominator is a power of $1-3t$. With  this in  mind, we define the multiplicative closed set
\begin{align*}
\mathcal{S}:=\{(1-3t)^n:n\in\mathbb{Z}_{n\geq0}\}.
\end{align*}
We will work with subspaces of  the localized ring
\begin{align*}
\mathbb{Z}[t,t^{-1}]_{\mathcal{S}}:=\mathcal{S}^{-1}\mathbb{Z}[t,t^{-1}].
\end{align*}
Given the form of $L_1$ in \eqref{L-1}, we expect to represent each $L_{\alpha}$ in the form $\frac{t^n}{(1-3t)^m}$ for $n\geq-1$ and $m\geq0$.
If we set
\begin{align*}
&L_{2\alpha-1}=p_0y^{3^{2\alpha}-1}\sum_{n\geq-1}d_n^{(2\alpha-1)}t^n,\\
&L_{2\alpha}=p_1y^{3^{2\alpha+1}-3}\sum_{n\geq0}d_n^{(2\alpha)}t^n,
\end{align*}
for $\alpha\geq 1$, and  recall that $\pi(n)$ denotes the 3-adic order of n (i.e. the highest power of 3 that divides $n$), then Theorem \ref{main-result} follows from the next lemma.

\begin{lemma}\label{main-lemma}
For every integer $\alpha\geq1$, we have
\begin{align}
\label{L-odd}\pi(d_n^{(2\alpha-1)})\geq\lfloor\frac{2n+5}{3}\rfloor+\alpha, ~~\text{for $n\geq -1$},\\
\label{L-even}\pi(d_n^{(2\alpha)})\geq\lfloor\frac{2n+5}{3}\rfloor+\alpha, ~~\text{for $n\geq 1$}.
\end{align}
In particular,
\begin{align}
\label{odd-special}&\pi(d_0^{(2\alpha-1)})\geq\alpha+2,\\
\label{even-special}&\pi(d_0^{(2\alpha)})\geq\alpha+2.
\end{align}
\end{lemma}

\begin{proof}
We proceed by induction on $\alpha$. From \eqref{L-1}, we have
\begin{align*}
L_1=&y^8p_0(4\cdot3^2t^{-1}+71\cdot3^4+2351\cdot3^3t+89\cdot3^5t^2-1975\cdot3^4t^3+407\cdot3^5t^4\\
\nonumber&-19\cdot3^7t^5+11\cdot3^7t^6-3^8t^7),
\end{align*}
which shows that $\pi(d_n^{(1)})\geq\lfloor\frac{2n+5}{3}\rfloor+1$ and $\pi(d_0^{(1)})=4$. Thus, \eqref{L-odd} and \eqref{odd-special} hold for $\alpha=1$.

Next,  suppose \eqref{L-odd} and \eqref{odd-special} hold for $\alpha\leq s$. We want to use $U_B$ operator on $L_{2s-1}$ to get $L_{2s}$, especially the coefficients $d_n^{(2s)}$ for $n\geq0$. Then by Definition \ref{r:defL} and Lemma \ref{U-fundamental}, we have
\begin{align*}
L_{2s}=&U_B(L_{2s-1})\\
=&\sum_{k\geq-1}d_k^{(2s-1)}U_B(p_0y^{3^{2s}-1}t^k)\\
=&p_1y^{3^{2s+1}-3}\sum_{k\geq-1}d_k^{(2s-1)}\sum_{n\geq \lceil\frac{k-1}{3}\rceil}b(k,3^{2s}-1,n)t^n\\
=&p_1y^{3^{2s+1}-3}\sum_{n\geq0}\sum_{-1\leq k\leq 3n+1}d_k^{(2s-1)}b(k,3^{2s}-1,n)t^n .
\end{align*}
Comparing the coefficients of $p_1y^{3^{2s+1}-3}t^n$ on both sides, we have
\begin{align}
\label{coeff-even}d_{n}^{(2s)}=\sum_{-1\leq k\leq 3n+1}d_k^{(2s-1)}b(k,3^{2s}-1,n).
\end{align}
We notice that $\pi(n)$ has the properties that $\pi(x+y)\geq\min\{\pi(x),\pi(y)\}$ and $\pi(xy)=\pi(x)+\pi(y)$ for $x,y\in\mathbb{Z}$. According to \eqref{U-0-p1t} and \eqref{L-odd}, we obtain that $\pi(b(k,3^{2s}-1,n))\geq \lfloor\frac{2n-k+2}{3}\rfloor$ and $\pi(d_k^{(2s-1)})\geq \lfloor\frac{2k+5}{3}\rfloor+s$.
Hence, for $n\geq 1$,
\begin{align*}
\pi(d_n^{(2s)})&\geq\min_{k\geq-1}\{\pi(d_k^{(2s-1)})+\pi(b(k,3^{2s}-1,n))\}\\
&\geq\min_{k\geq-1}\{\lfloor\frac{2k+5}{3}\rfloor+s+\lfloor\frac{2n-k+2}{3}\rfloor      \}\\
&\geq \lfloor\frac{2n+5}{3}\rfloor+s,
\end{align*}
which indicates that \eqref{L-even} holds for $\alpha=s$.

For $d_0^{(2s)}$, from \eqref{coeff-even}, we obtain that
\begin{align*}
d_0^{(2s)}=&d_{-1}^{(2s-1)}b(-1,3^{2s}-1,0)+d_0^{(2s-1)}b(0,3^{2s}-1,0)+d_1^{(2s-1)}b(1,3^{2s}-1,0).
\end{align*}
By the induction hypothesis, $\pi(d_{-1}^{(2s-1)})\geq s+1$, $\pi(d_{0}^{(2s-1)})\geq s+2$, and $\pi(d_{1}^{(2s-1)})\geq s+2$.

Let $[t^k]f$  denote the coefficients of $t^k$ in $f$. Using\eqref{re-y}, we find
\begin{align}\label{coeff-eg}
b(-1,m,0)
=&[p_1y^{3m}]U_B(p_0y^mt^{-1})\\
\nonumber=&[p_1y^{3m}]\{y^9(1-2\cdot3^2t+5\cdot3^3t^2-20\cdot3^3t^3+5\cdot3^5t^4-2\cdot3^6t^5+3^6t^6)\\
\nonumber&\cdot U_B(p_0y^{m-3}t^{-1})+y^6(-3+3^2t+3^4t^2-14\cdot3^3t^4+3^6t^5-3^6t^6)\\
\nonumber&\cdot U_B(p_0y^{m-2}t^{-1})+y^3(3+4\cdot3^2t+3^3t^2-2\cdot3^3t^3)U_B(p_0y^{m-1}t^{-1})\}\\
\nonumber=&b(-1,m-3,0)-3b(-1,m-2,0)+3b(-1,m-1,0).
\end{align}
From the Appendix,  $b(-1,0,0)=12$, $b(-1,1,0)=36$, and $b(-1,2,0)=69$, all  divisible by $3$. By the recurrence  above,  $3\mid b(-1,m,0)$ for all $m\in\mathbb{N}$. Hence, \eqref{even-special} holds for $\alpha=s$.

Now suppose \eqref{L-even} and \eqref{even-special} hold for $\alpha\leq s$. Then by Definition \ref{r:defL} and Lemma \ref{U-fundamental},
\begin{align*}
L_{2s+1}=&U_A(L_{2s})\\
=&\sum_{k\geq0}d_k^{(2s)}U_A(p_1y^{3^{2s+1}-3}t^k)\\
=&p_0y^{3^{2s+2}-1}\sum_{k\geq0}d_k^{(2s)}\sum_{n\geq \lceil\frac{k-3}{3}\rceil}a(k,3^{2s+1}-3,n)t^n\\
=&p_0y^{3^{2s+2}-1}\sum_{n\geq-1}\sum_{0\leq k\leq 3n+3}d_k^{(2s)}a(k,3^{2s+1}-3,n)t^n .
\end{align*}
Equating  coefficients of $p_0y^{3^{2s+2}-1}t^n$ on both sides, we get
\begin{align}
\label{coeff-odd}d_{n}^{(2s+1)}=\sum_{0\leq k\leq 3n+3}d_k^{(2s)}a(k,3^{2s+1}-3,n).
\end{align}
According to \eqref{U-0-t}, \eqref{L-even}, and \eqref{even-special}, we obtain that $\pi(a(k,3^{2s+1}-3,n))\geq\lfloor\frac{2n-k+3}{2}\rfloor$, $\pi(d_k^{(2s)})\geq\lfloor\frac{2n+5}{3}\rfloor+s$ for $k\geq1$, and $\pi(d_0^{(2s)})\geq s+2$.
Thus, for $n\geq -1$,
\begin{align*}
\pi(d_n^{(2s+1)})&\geq\min_{k\geq0}\{\pi(d_k^{(2s)})+\pi(a(k,3^{2s+1}-3,n))\}\\
&\geq\min\{\min_{k\geq1}\{\pi(d_k^{(2s)})+\pi(a(k,3^{2s+1}-3,n))\},\pi(d_0^{(2s)})+\pi(a(0,3^{2s+1}-3,n))\}\\
&\geq\min\{\min_{k\geq1}\{\lfloor\frac{2k+5}{3}\rfloor+s+\lfloor\frac{2n-k+3}{3}\rfloor\}, s+2+\lfloor\frac{2n+3}{3}\rfloor \}\\
&\geq\min\{\min\{\min_{k\geq2}\{\lfloor\frac{2n+k+3}{3}\rfloor+s+1\},\lfloor\frac{2n+5}{3}\rfloor+s+1\},\lfloor\frac{2n+6}{3}\rfloor+s+1\}\\
&\geq \lfloor\frac{2n+5}{3}\rfloor+s+1,
\end{align*}
where the forth inequality is obtained by splitting the part where $k\textgreater 1$ into two cases: $k = 1$ and $k\geq 2$.
Thus, \eqref{L-odd} holds for $\alpha=s+1$.

From \eqref{coeff-odd}, we have
\begin{align*}
d_0^{(2s+1)}=&d_{0}^{(2s)}a(0,3^{2s+1}-3,0)+d_1^{(2s)}a(1,3^{2s+1}-3,0)+d_2^{(2s)}a(2,3^{2s+1}-3,0)\\
&+d_3^{(2s)}a(3,3^{2s+1}-3,0).
\end{align*}
By the induction hypothesis,  $\pi(d_{0}^{(2s)})\geq s+2$, $\pi(d_{1}^{(2s)})\geq s+2$, $\pi(d_{2}^{(2s)})\geq s+3$, and $\pi(d_{3}^{(2s)})\geq s+3$.

Similar to \eqref{coeff-eg}, we derive
\begin{align*}
a(0,m,0)=&a(0,m-3,0)-18a(0,m-3,-1)-3a(0,m-2,0)+9a(0,m-2,-1)\\
&+3a(0,m-1,0)+36a(0,m-1,-1),\\
a(1,m,0)=&a(1,m-3,0)-3a(1,m-2,0)+3a(1,m-1,0).
\end{align*}
From the Appendix,
\begin{align*}
&a(0,0,0)=285,\hspace{10pt}  a(0,1,0)=600,\hspace{10pt}a(0,2,0)=1068;\\
&a(1,0,0)=66,\hspace{16pt} a(1,1,0)=108,\hspace{10pt}a(1,2,0)=159.
\end{align*}
Hence,  $3\mid a(0,m,0)$ and $3\mid a(1,m,0)$ for all $m\in\mathbb{N}$, so  \eqref{odd-special} holds for $\alpha=s+1$. This completes the proof of Lemma \ref{main-lemma}.
\end{proof}

\section*{Appendix. The fundamental relations}

Group I:
\begin{align*}
\nonumber
y=\frac{1}{1-3t}, \hspace{10pt}p_0=(1+t)^4,\hspace{10pt}p_1=(1+t)^2,\hspace{10pt}L_0=t^{-1}+3^3+3t+3^2t^2.
\end{align*}

Group II:
\begin{align*}
U_A(p_1t^{-1})=&y^8p_0(11t^{-1}+38\cdot3^3+1085\cdot3^2t+212\cdot3^3t^2-961\cdot3^3t^3+98\cdot3^4t^4+3^7t^5);\\
U_A(p_1)=&y^8p_0(t^{-1}+95\cdot3+16\cdot3^5t+5\cdot3^3t^2-308\cdot3^3t^3+91\cdot3^4t^4-20\cdot3^5t^5+11\cdot3^5t^6-3^6t^7);\\
U_A(p_1t)=&y^8p_0(22\cdot3+176\cdot3^2t-11\cdot3^4t^2-19\cdot3^4t^3+8\cdot3^4t^4+79\cdot3^4t^5-88\cdot3^5t^6+17\cdot3^7t^7\\
&-2\cdot3^9t^8+11\cdot3^7t^9-3^8t^{10});\\
U_A(p_1yt^{-1})=&y^{11}p_0(14t^{-1}+101\cdot3^3+7760\cdot3^2t+188\cdot3^7t^2+9572\cdot3^3t^3-16102\cdot3^4t^4-128\cdot3^7t^5\\
&+20\cdot3^{10}t^6-14\cdot3^9t^7-3^{10}t^8);\\
U_A(p_1y)=&y^{11}p_0(t^{-1}+200\cdot3+284\cdot3^4t+5752\cdot3^3t^2+1234\cdot3^3t^3-5912\cdot3^4t^4+884\cdot3^5t^5\\
&+392\cdot3^{5}t^6-55\cdot3^6t^7);\\
U_A(p_1yt)=&y^{11}p_0(4\cdot3^3+802\cdot3^2t+751\cdot3^4t^2-224\cdot3^4t^3-1892\cdot3^4t^4+2116\cdot3^4t^5\\
&-514\cdot3^{5}t^6+16\cdot3^8t^7-32\cdot3^7t^8+14\cdot3^7t^9-3^8t^{10});\\
U_A(p_1y^2t^{-1})=&y^{14}p_0(17t^{-1}+215\cdot3^3+31742\cdot3^2t+148039\cdot3^3t^2+657251\cdot3^3t^3+119546\cdot3^4t^4\\
&-29324\cdot3^7t^5-1930\cdot3^9t^6+4745\cdot3^9t^7+241\cdot3^{10}t^8-10\cdot3^{14}t^9+17\cdot3^{12}t^{10}+3^{13}t^{11});\\
U_A(p_1y^2)=&y^{14}p_0(t^{-1}+356\cdot3+2963\cdot3^3t+17042\cdot3^4t^2+239206\cdot3^3t^3+16864\cdot3^4t^4\\
&-32678\cdot3^6t^5+692\cdot3^5t^6+35045\cdot3^6t^7-52\cdot3^{11}t^8-17\cdot3^{11}t^9+2\cdot3^{12}t^{10});\\
U_A(p_1y^2t)=&y^{14}p_0(53\cdot3+770\cdot3^3t+5827\cdot3^4t^2+29704\cdot3^4t^3-1958\cdot3^5t^4\\
&-103100\cdot3^4t^5+21722\cdot3^5t^6+392\cdot3^8t^7-967\cdot3^7t^8+98\cdot3^7t^9-3^8t^{10}).
\end{align*}

Group III:
\begin{align*}
U_B(p_0t^{-1})=&y^{0}p_1(4\cdot3-14\cdot3t+11\cdot3^2t^2-3^4t^3-2\cdot3^4t^4+5\cdot3^4t^5-3^5t^6);\\
U_B(p_0)=&y^{0}p_1(5-11\cdot3t+2\cdot3^4t^2-14\cdot3^3t^3-3^3t^4+8\cdot3^5t^5-13\cdot3^5t^6+5\cdot3^6t^8-3^7t^9);\\
U_B(p_0t)=&y^{0}p_1(1-3^2t+14\cdot3^2t^2-28\cdot3^3t^3+61\cdot3^3t^4+23\cdot3^4t^5-20\cdot3^6t^6+8\cdot3^7t^7\\
&+25\cdot3^6t^8-23\cdot3^7t^9+2\cdot3^8t^{10}+5\cdot3^8t^{11}-3^9t^{12});\\
U_B(p_0yt^{-1})=&y^{3}p_1(4\cdot3^2+88\cdot3t-55\cdot3^2t^2+8\cdot3^4t^3-10\cdot3^4t^4+8\cdot3^4t^5-3^5t^6);\\
U_B(p_0y)=&y^{3}p_1(8+46\cdot3t-16\cdot3^3t^2+38\cdot3^3t^3-64\cdot3^3t^4+5\cdot3^5t^5+8\cdot3^5t^6-8\cdot3^6t^7\\
&+8\cdot3^6t^8-3^7t^9);\\
U_B(p_0yt)=&y^{3}p_1(1+8\cdot3^2t-38\cdot3^2t^2+16\cdot3^4t^3-128\cdot3^3t^4+56\cdot3^4t^5+2\cdot3^7t^6-40\cdot3^6t^7\\
&+64\cdot3^6t^8-8\cdot3^7t^9-2\cdot3^9t^{10}+8\cdot3^8t^{11}-3^9t^{12});\\
U_B(p_0y^2t^{-1})=&y^{6}p_1(23\cdot3+802\cdot3t+923\cdot3^2t^2-140\cdot3^4t^3-7\cdot3^5t^4+38\cdot3^4t^5-3^5t^6);\\
U_B(p_0y^2)=&y^{6}p_1(11+274\cdot3t+14\cdot3^5t^2-31\cdot3^5t^3+314\cdot3^3t^4-49\cdot3^5t^5+62\cdot3^5t^6-19\cdot3^6t^7\\
&+11\cdot3^6t^8-3^7t^9);\\
U_B(p_0y^2t)=&y^{6}p_1(1+28\cdot3^2t+20\cdot3^4t^2-197\cdot3^3t^3+430\cdot3^3t^4-280\cdot3^4t^5+5\cdot3^8t^6-32\cdot3^6t^7\\
&-32\cdot3^6t^8+40\cdot3^7t^9-17\cdot3^8t^{10}+11\cdot3^8t^{11}-3^9t^{12}).
\end{align*}

\subsection*{Acknowledgements}
The first author was  supported by the National Key R\&D Program of China (Grant No. 2024YFA1014500) and the National Natural Science Foundation of China (Grant No. 12201387).






\end{document}